\documentclass[11pt, a4paper]{amsart}
\usepackage[utf8]{inputenc}
\usepackage{amsmath,amsthm,verbatim,amscd,amssymb,setspace,enumitem,hyperref}
\usepackage[T1]{fontenc}
\usepackage{amsfonts}
\usepackage[english]{babel}
\usepackage{array,epsfig}
\usepackage{amsxtra}
\usepackage{latexsym}
\usepackage{dsfont}
\usepackage{mathrsfs}
\usepackage[mathscr]{eucal}
\usepackage{color}
\usepackage[all]{xy}
\usepackage{hyperref}
\usepackage{graphicx}
\usepackage{mathtools}
\usepackage{caption}
\usepackage{relsize}
\usepackage{todonotes}
\usepackage{fancyvrb}
\usepackage{xcolor}
\usepackage{tikz-cd}

\theoremstyle{plain}
\newtheorem{defn}{Definition}[section]
\newtheorem{thm}[defn]{Theorem}
\newtheorem*{theorem*}{Theorem}

\newcounter{mtheorem}
\newtheorem{mtheorem}[mtheorem]{Theorem}

\theoremstyle{definition}
\newtheorem{cor}[defn]{Corollary}
\newtheorem{rmk}[defn]{Remark}

\newtheorem{lem}[defn]{Lemma}
\newtheorem{ex}[defn]{Example}

\newtheorem{prop}[defn]{Proposition}

\newcommand{\set}[1]{\left\{#1\right\}}
\newcommand{\tuple}[1]{\left(#1\right)}

\newcommand{\abs}[1]{\left|#1\right|}

\newcommand{\sprod}[1]{\left<#1\right>}

\newcommand{\ol}[1]{\overline{#1}}

\newcommand{\wt}[1]{\widetilde{#1}}

\newcommand{\pfrak}{\mathfrak{p}}

\newcommand{\tfrak}{\mathfrak{t}}

\newcommand{\Cbb}{\mathbb{C}}

\newcommand{\Nbb}{\mathbb{N}}

\newcommand{\Pbb}{\mathbb{P}}
\newcommand{\Qbb}{\mathbb{Q}}
\newcommand{\Rbb}{\mathbb{R}}

\newcommand{\Zbb}{\mathbb{Z}}

\newcommand{\Bcal}{\mathcal{B}}
\newcommand{\Ccal}{\mathcal{C}}

\newcommand{\Fcal}{\mathcal{F}}

\newcommand{\Ical}{\mathcal{I}}

\newcommand{\Kcal}{\mathcal{K}}
\newcommand{\Lcal}{\mathcal{L}}

\newcommand{\Ocal}{\mathcal{O}}

\newcommand{\Xcal}{\mathcal{X}}
\newcommand{\Ycal}{\mathcal{Y}}
\newcommand{\Ncal}{\mathcal{N}}

\renewcommand{\phi}{\varphi}

\newcommand{\del}{\partial}
\newcommand{\delb}{\overline{\partial}}

\newcommand{\Scal}{\operatorname{Scal}}

\title{K-semistability and singularities of normal affine cones}

\usepackage[doi=false, url=false, style=alphabetic, isbn=false]{biblatex}
\author{Tran-Trung Nghiem}
\begin{document}
\maketitle

\begin{abstract}
We provide a complete and optimal characterization of $\Qbb$-Gorenstein normal affine singularities in terms of their K-stability in the sense of Collins--Székelyhidi. 
\end{abstract}


\section{Introduction}

\subsection{Main result}

Let \( Y \) be a  $T$-variety, i.e. a complex normal affine variety of dimension \( n+1\) endowed with the effective action of a torus \( T \) \cite{AH06}. When the action is good, we say that $(Y,T)$ is a \emph{normal affine cone}. In concrete terms, this is the cone over a projective variety with a given embedding in some projective space. 
In the last 20 years, 
a notion of K-stability for $Y$ has been well developed and refined notably in the works \cite{MSY06, MSY08, RT11, CS18, CS19, LWX, Wu22}. This condition is expected to characterize the existence of constant scalar curvature Sasaki metrics on the cone's base. The metrics have positive scalar curvature if and only if it lifts to a scalar-flat Kähler cone metric on $Y$, and conversely any cscK cone metric on $Y$ must be scalar-flat. In complex geometry, affine cones appear, for example, as local tangent cones of isolated singularities; or asymptotic cones of complete Kähler metrics. 
Thus, understanding how K-(semi)stability affects cone singularities would be the first step to understand more general cscK metrics near local singularities.

On the other hand, seminal results \cite{Oda12, Oda13, OS15} in the compact cases show that the K-stability notion, defined in \cite{Tia97, Don02}, must surprisingly restrict singularities of a scheme (with $\Qbb$-Cartier canonical divisor) to the classes originally introduced to run Mori's MMP; cf. \cite{KSB88, Ale96, KM98} for a detailed account. 

Recall from \cite{Oda13} that if a normal compact polarized variety $(X,L)$ is K-semistable, then $X$ must have \emph{log canonical singularities} (or semi-log canonical if $X$ is not normal); and \emph{Kawamata log terminal singularities} (klt) when $L = -K_X$. Namely, let $K_X$ be the canonical divisor on $X$ such that $mK_X$ is Cartier for some $m > 0$; then $X$ is said to be log canonical if for one (hence any) resolution $X' \to X$, 
\[ K_{X'} \sim_{\Qbb} f^*(K_X) + \sum_{i} a_iE_i, \]
the discrepancy  $a_i = a(X,E_i)$ satisfies $a_i \geq -1$ (resp. $> -1$). The definition applies as well for an arbitrary normal variety $Y$ with $mK_Y$ Cartier (i.e., $Y$ is $\Qbb$-Gorenstein) so that $f^{*}mK_Y$ is well-defined. 
However, it turns out that the cone situation is slightly more subtle as the following examples show. 
\begin{ex}
If $X$ is a projective variety with ample canonical divisor $K_X$ and klt singularities (hence K-stable \cite{Oda12}), then $C(X,K_X)$ is $\Qbb$-Gorenstein and K-stable, but does not have klt singularities \cite[Lemma 3.1]{Kol13}. 

Another situation is when $(X,L)$ is a polarized pair and $X$ is Calabi--Yau klt, then $C(X,L)$ is again $\Qbb$-Gorenstein and K-stable but may not even have rational singularities \cite[Lemma 3.6]{Kol13}, hence not log terminal. 
\end{ex}
Thus, K-semistability for a normal affine cone would at best allow us to deduce log canonical singularities outside of the apex. On the other hand, when the variety is an affine cone over a log Fano variety, then K-semistability should also imply that the apex is a klt singularity as well; that is, any exceptional divisor dominating the apex must have log discrepancy $> -1$. A $\Qbb$-Gorenstein normal affine cone $Y$ with klt singularities also goes under the name of \emph{Fano cone singularity} \cite{LWX}. This type of singularity is assumed to be the right category for Calabi--Yau metrics on normal affine cones \cite{MSY08, CS19}. Our main theorem confirms this expectation and provides an answer to \cite[Problem 2.19]{Oda24a}. 

\begin{mtheorem} \label{theorem:main}
Let $(Y,T)$ be a $\Qbb$-Gorenstein normal affine cone. Then if $(Y,T, \xi)$ is K-semistable, the K-semistable Reeb vector is unique up to scaling and $Y \backslash \set{0}$ has log canonical singularities. More precisely, let $a_1(\xi)$ be the slope of the K-semistable pair $(Y,\xi)$  (cf. Proposition \ref{prop:indexcharexpansion}), we have the following trichotomy: 
\begin{enumerate}
    \item \label{case:gentype} If $a_1(\xi) < 0$ then  $(Y,T,\xi)$ is K-stable iff $Y \backslash \set{0}$ is log canonical.  
    \item \label{case:calabiyau} If $a_1(\xi) = 0$ then $(Y,T,\xi)$ is K-stable iff $Y \backslash \set{0}$ is klt. 
    \item If $a_1(\xi) > 0$ and $(Y,T,\xi)$ is K-semistable, then $Y$ is itself klt; i.e., $(Y,T)$ is a Fano cone singularity.    
\end{enumerate}
In case \ref{case:gentype} (resp. \ref{case:calabiyau}), $(Y,T)$ is in fact the cone over a log pair of general type (resp. log Calabi--Yau polarized variety). The Reeb cone is one-dimensional in both cases.  
\end{mtheorem}

Of course a similar result also holds for a (non-normal) affine cone pair $(Y,D)$, where $D$ is an effective Weil divisor such that $K_Y + D$ is $\Qbb$-Cartier. Since the proof can be repeated without any extra difficulty, we omit it for convenience and only note here some metric interpretation. 

When $(Y,D)$ is a pair of log Fano cone singularity, K-stability is equivalent to $(Y,D)$ admitting Ricci-flat Kähler cone metrics on $Y \backslash D$ with cone singularities over $D$. This has been known from the works \cite{MSY08, CS19, Li21, Hua22} (see also \cite{dBL22, Ngh25} when the log pair has a spherical action, and \cite{GP16, OS15} for the compact case). Moreover, it is also known that it suffices to test K-stability of log Fano cone singularities on special test configurations \cite{LW25}. 

When the cone is not a log Fano singularity but still $\Qbb$-Gorenstein, our result implies that K-stability is equivalent to the (singular) Sasaki base admitting Sasaki--Einstein metrics of the same Einstein constant as the slope sign. This has been known when $Y \backslash \set{0}$ is smooth, in which case the Reeb cone is one-dimensional \cite[Theorem 8.1.14]{BG08}, hence K-stability of $Y$ boils down to K-stability of a log Calabi--Yau pair or a log pair of general type.  

In general, the stability of a cone pair $(Y,D)$---not assumed to be $\Qbb$-Gorenstein, should be related to existence of constant scalar curvature Sasaki metrics on the (singular) Sasaki link of $Y$ with conical singularities along the Sasaki base of $D$; see \cite{CS18, ACL21} for some partial results in the metric-to-stability direction. The K-stability notion for a pair also goes under the name of log K-stability \cite{OS15, AHZ25}.



\subsection{Towards singularities of cscK cones}
Throughout the proof of Theorem \ref{theorem:main}, we have crucially used the $\Qbb$-Gorenstein condition to define discrepancy \cite{KM98}. As it turns out, K-semistability can still be defined without the $\Qbb$-Gorenstein assumption \cite{CS18} (see also \cite{Wu22} for a definition of Monge--Ampère energy for an arbitrary polarized cone). Moreover, K-semistability is implied by the existence of cscK cone metrics on $Y \backslash \set{0}$ \cite[Corollary 1.1]{CS18}. Examples of non-$\Qbb$-Gorenstein cscK toric cones have been built by E. Legendre \cite{Leg11}. An interesting question is to see whether K-stability of a general cone imposes restrictions on singularities in a more general sense.  


Note that when $Y$ is a Ricci-flat cone in the sense of \cite{EGZ}, then $\Qbb$-Gorenstein follows from van Coevering's theorem when $Y$ has an isolated singularity \cite{vC10}. A recent work of Hallgren--Székelyhidi shows that this also holds when the singularities are defined in a reasonable sense that includes the klt case  \cite{HS25}. In contrast with these results, 
our following remark suggests that a K-stable cone with cscK but not Ricci-flat metric is never $\Qbb$-Gorenstein, no matter how we define a priori singular cscK metrics so that they coincide with the klt notion in \cite{EGZ}.

\begin{thm} \label{theorem:cscK_gorenstein}
Assume that $(Y,T)$ is a K-stable $\Qbb$-Gorenstein cone with K-stable Reeb vector $\xi$. Assume also that any pair of cscK cone metrics with the same Reeb vector $\xi$ is related by an automorphism preserving the Reeb vector. 
Then any cscK cone metric on $Y \backslash \set{0}$ with Reeb vector $\xi$ is Ricci-flat.  
\end{thm}

\begin{rmk}
Here, our arguments rely on a strong uniqueness modulo automorphisms assumption. The latter holds when the Sasaki link is smooth (i.e., the cone has an isolated singularity) \cite{BB17, HL21}. We expect this uniqueness property to also hold in the singular case (details will appear elsewhere).     
\end{rmk}

In the toric examples built in \cite{Leg11}, $K_Y$ is not $\Qbb$-Cartier but $K_Y +D$ is always $\Qbb$-Cartier for a suitable choice of an effective boundary $D$, which can be modified so that $(Y,D)$ is log canonical \cite{CLS11}. 
Based on this evidence, we expect the following: \emph{if $Y$ is K-semistable, then there is an effective $\Qbb$-boundary $D$ such that $(Y,D)$ is log canonical (in particular $(K_Y+D)$ is $\Qbb$-Gorenstein)}. In this situation, the pair $(Y,Y_{sing})$ is said to be log canonical in the sense of de Fernex--Hacon \cite[Definition 7.1]{dFH09}.



\subsection{Organization}
In Section \ref{section:cones}, we recall the definition of a normal $\Qbb$-Gorenstein affine cone and several of its properties needed in our paper. We then recall their K-stability notion in the sense of Collins--Székelyhidi. A crucial result in this section (Theorem \ref{theorem:conekstab_vs_weightedkstab}) is the correspondence between test configurations of a $\Qbb$-Gorenstein cone, and the equivariant test configurations of an arbitrary quasi-regular quotient. 

Finally, in Section \ref{section:proof_main}, we prove Theorem \ref{theorem:main} and justify Theorem \ref{theorem:cscK_gorenstein}. Besides some technical points, our main strategy is inspired by Odaka's \cite{Oda13}, with the main ingredient being the existence of an \emph{equivariant} log canonical model \cite{OX12}. 
\\
\textbf{Acknowledgement.} 
I am thankful to Yuji Odaka for helpful suggestions, and E. Legendre, R. Reboulet, S. Jubert, Y. Wu, V. Guedj, T-D. To and C. Pan; the discussions with whom led to improvements on various parts of the manuscripts. The author is supported by the ANR–FAPESP-21-CE40-0017 project BRIDGES. 

\section{Affine cones} \label{section:cones}

\subsection{Definitions and examples}

Let $Y$ be a normal affine variety of complex dimension $n+1$ with an effective torus action of $T$. 
Let $R$ be the coordinate ring of $Y$; and $M = \text{Hom}(T, \Cbb^{*})$ (resp. $N = M^{\vee}$) be the weight lattice (resp. coweight lattice) of $T$ so that $M \simeq \Zbb^d$; also denote by $M_{\Rbb}, N_{\Rbb} \coloneqq M \otimes_{\Zbb} \Rbb, N \otimes_{\Zbb} \Rbb$. 

Under the $T$-action, $R$ has the following decomposition simple $T$-modules 
\begin{equation}
 R = \oplus_{\alpha} R_{\alpha},  
\end{equation}
where the set $\set{\alpha \in M, R_{\alpha} \neq 0}$ has a finite generator set $S$, the cone generated by which is called the \emph{moment cone} of $(Y,T)$, denoted by $\sigma$. The interior of $\sigma^{\vee}$ is called the \emph{Reeb cone} of $(Y,T)$, denoted by $C_R$. 
The terminology is inspired by the differentio-geometric notion of Reeb vector fields for Kähler cones; see \cite{ACL21} for a comprehensive treatment from this point of view.


We say that the $T$-action on $Y$ is \emph{good} if it acts effectively with a unique fixed point $0$, contained in the closure of any $T$-orbit. The fixed point is often called an attractive fixed point in the literature \cite{PS11}. A
good action can be characterized as follows. 

\begin{prop}[{\cite{PS11, Oda24a}}] \label{prop:goodaction}
Let $Y$ be a $T$-variety. The following are equivalent
\begin{itemize}
    \item The $T$-action is good.
    \item The moment cone $\Rbb_{\geq 0} \sigma$ is strictly convex of maximal dimension in $M_{\Rbb}$. 
    \item The Reeb cone $C_R$ is strictly convex of maximal dimension (hence non-empty). 
\end{itemize}
\end{prop}

 This characterization, together with Altmann--Hausen theory \cite{AH06}, allows us to easily produce examples of affine varieties without good $T$-action. The simplest case is of course to take $T = \Cbb^{*}$ and $\sigma = \set{0}$; cf.\cite[Section 4.3]{PS11}. 

Recall that for a normal variety $Y$ with inclusion of the regular locus $\iota \colon Y^{\text{reg}} \hookrightarrow Y$, we can define a canonical sheaf to be 
\[ \Kcal_Y \coloneqq \iota_{*}\Kcal_{Y^{\text{reg}}}, \]
which extends to an invertible sheaf over $Y$ by Hartogs' lemma, so there is a Weil divisor $K_Y$ such that $\Kcal_Y = \Ocal(K_Y)$, which we call the \emph{canonical divisor} of $Y$. 

\begin{defn} \label{defn:normalaffinecone}
\hfill
\begin{itemize}
\item (Normal affine cone) Given a normal affine variety $Y$ with a good $T$-action,
we call the pair $(Y,T)$ a \emph{normal affine cone}. A pair decorated with a Reeb vector $(Y,T,\xi)$ is a \emph{polarized cone}. 

\item ($\Qbb$-Gorenstein cone) The normal affine cone is said to be \emph{$\Qbb$-Gorenstein} if $K_Y$ is $\Qbb$-Cartier; namely there is $m \in \Nbb$ such that $mK_Y$ is a Cartier divisor. A $\Qbb$-Gorenstein cone with klt singularities is called a \emph{Fano cone singularity}.

\item We say that $\xi$ is \emph{quasi-regular} if there is $c \in \Rbb$ such that $c \xi \in \Nbb^k$; otherwise $\xi$ is said to be \emph{irregular}. Equivalently, $\xi$ is quasi-regular if it generates a $S^1$-action of rank $1$; if this action is moreover free, we say that $\xi$ is \emph{regular}.
\end{itemize}
\end{defn}

\begin{rmk}
\hfill
\begin{enumerate}
\item In \cite[Definition 2.9 (iii)]{Oda24a}, a $\Qbb$-Gorenstein cone is defined to be a reduced scheme which is normal crossing in codimension $1$, satisfying Serre's $S_2$ condition, with $\Qbb$-Cartier canonical bundle. The Serre's $S_2$ condition is implied by our normality assumption ; see \cite[Chap. IV.4]{Ser65}. 

\item We distinguish between the conditions \emph{Gorenstein} and \emph{$\Qbb$-Gorenstein of index $1$}. 
The former means that $K_Y$ is Cartier and $Y$ is Cohen--Macaulay (CM) (cf. \cite[Definition 2.1.1]{BH93}); while the latter only means that $K_Y$ is Cartier. A $\Qbb$-Gorenstein klt singularity is CM (in particular a Fano cone singularity is CM); but a cone over a polarized abelian variety of dimension at least $2$ is not CM, although it is $\Qbb$-Gorenstein of index $1$. We refer the reader to \cite[Corollary 3.4]{Kol13} for several CM criteria and examples. 
\end{enumerate}
\end{rmk}

\begin{ex}
Recall that a Riemannian cone $C$ is a Riemannian manifold $\Rbb^{+} \times L$ with the metric 
\[ g_c \coloneqq dr^2 + r^2 g_L, \]
$(L,g_L)$ being compact Riemannian manifold. It is called a \emph{Kähler cone} if there is a complex structure $J_c$ on $C$ such that $\omega_c = g(.,J_c.) = \frac{i \del \delb}{2} r^2$ is Kähler. 
Then the completion $Y$ of $C$ by any Kähler cone metric $g$ has a unique structure of \emph{normal affine variety} with a $T$-action \cite[Theorem 3.1]{vC11}; namely, there is a $T$-equivariant embedding of $Y \hookrightarrow \Cbb^N$ such that $T$ acts on the embedding by diagonal matrices. Moreover, the moment cone in the algebraic sense $\sigma$ can be identified with the symplecto-geometric moment cone under the action of the maximal compact torus $T_c \subset T$ \cite{CS18}, which is always full-dimensional.

Conversely, any normal affine cone $(Y,T)$ such that $Y \backslash \set{0}$ is smooth can be realized as a Kähler cone $(C,g_c,J_c)$ so that $C$ is diffeomorphic to $Y \backslash \set{0}$. In fact by normality of $Y$, given a quasi-regular Reeb vector $\xi_0$ on $Y$ we can embed $Y$ in a $T$-equivariant manner inside $\Cbb^N$, so that $T$ acts diagonally on $\Cbb^N$ \cite{Sum74} and the $\xi_0$-action is given by 
\[ \xi_0(z_1,\dots,z_N) = (e^{iw_i}z_1,\dots, e^{iw_N}z_N).\]
We can then define a Kähler cone metric on $Y \backslash \set{0}$ by setting 
\[ \omega \coloneqq \frac{i \del \delb}{2} r^2, \quad r^2 \coloneqq \sum_{i=1}^N \abs{z_i}^2.\]
\end{ex}




On the other hand, we can view a normal affine cone as the total space of polarizations over a projective variety. 
To be precise, let $X$ be a projective variety with a very ample line bundle $L$; then the \emph{affine cone} with respect to $(X,L)$ is defined as 
\begin{equation}
 Y \coloneqq C(X,L) \coloneqq \text{Spec} \bigoplus_k H^0(X,kL).    
\end{equation}
Further, given an embedding $X \hookrightarrow \Pbb^N(H^0(X,L))$, we can define $X$ as the zero locus of a family of homogeneous polynomials $f_1,\dots,f_k$; and $C(X,L)$ is the zero locus inside $\Cbb^{N+1}$. 

\begin{thm} \label{theorem:cone-base-singularities}
Let $Y$ be a normal affine cone and $\xi_0$ be a quasi-regular Reeb vector. Then the GIT quotient $Y - \set{0} / \xi_0$ has the structure of a normal polarized pair $(X,B;L)$, where $X$ is the ambient normal projective variety, $B = \sum b_i B_i$ is an effective $\Qbb$-Weil divisor on $X$ with $0\leq b_i < 1$, and $L \to X$ is an ample line bundle such that
\[ Y = C(X,L). \]
Moreover, let $D$ be the pullback divisor of $B$ on $Y$; then $(Y,D)$ is $\Qbb$-Gorenstein iff $\gamma L \sim_{\Qbb} K_{X,B}$ for some $\gamma \in \Qbb$. In which case, 
\begin{itemize} 
    \item $(Y,D)$ is log canonical iff $\gamma \leq 0$ and $(X,B)$ is log canonical. 
    \item $(Y,D)$ is klt iff $\gamma < 0$ (i.e., $(X,B)$ is a log Fano variety) and $(X,B)$ is klt. 
\end{itemize}
If, in addition, $K_Y$ is $\Qbb$-Cartier then the same assertions on singularities hold for $Y$. 
\end{thm}

\begin{proof}
The statement gathers known results in the literature \cite{RT11, Kol04, LL19, LWX, Kol13}. We record a proof here for convenience. 
Let $R$ be the coordinate ring of $Y$. 
 Since $\xi$ generates an effective $\Cbb^{*}$-action, $Y \backslash \set{0} \to X $  is equivariantly isomorphic to the total space (with zero section removed) of an orbiline bundle $L$ over an orbifold, whose underlying analytic space is homeomorphic to the variety $X = \text{Proj}(R)$. The bundle can be built  by taking $L \coloneqq X \times_{\Cbb^{*}} \Cbb$ (cf. \cite[Proposition 3.1.3]{Bri18}). Normality of $Y$ implies normality of $X$ \cite{MFK94}.
 
 In the terminology of \cite{OW75, Kol04}, $f\colon Y  \backslash \set{0} = L \backslash X \to X$ is a \emph{Seifert $\Cbb^{*}$-bundle} over $X$. The $\Cbb^*$-action generated by $\xi_0$ on every fiber $f^{-1}(p)$ is isomorphic to the $\Cbb^*$-action on $\Cbb^*/\mu_m$ for some $m = m(p,Y/X)$ where $\mu_m$ is the $m$-th group of unity. Moreover, 
 \begin{itemize}
 \item The set $\set{m(.,Y/X) > 1} \subset X$ is a closed analytic subset of $X$ which is union of Weil divisors $\cup_i B_i$ and a subset of codimension at least $2$, all contained in $X^{sing}$. 
 \item The multiplicity $m(.,Y/X)$ is constant on a dense subset of each $B_i$, whose common value $m_i \coloneqq m(B_i)$ is the multiplicity over $B_i$.
 \end{itemize}
 The $\Qbb$-Weil divisor 
 \[ B \coloneqq \sum (1- \frac{1}{m_i})B_i \]
is often called the \emph{branch divisor} of $X$.  

We have $\text{Pic}(Y) = 0$ and $\text{Cl}(Y) = \text{Cl}(X)/ \sprod{L}$ (cf. \cite[Proposition 3.14]{Kol13}); hence $K_{Y,D}$ is $\Qbb$-Cartier iff $ m K_{X,B} \sim_{\Qbb} \gamma L$ for some $\gamma \in \Nbb$. 
 Let $p \colon \wt{Y} \to Y$ be the blow-up of $Y$ at the vertex with exceptional divisor $E \simeq X$; and $\wt{D}$ be the pullback of $D$. If $K_{Y,D}$ is $\Qbb$-Cartier, then the line bundle $\pi \colon \wt{Y} \to X$ satisfies
\[ K_{\wt{Y}, \wt{D}} + E \sim_{\Qbb}  \pi^{*} K_{X,B} \sim_{\Qbb} \gamma \pi^{*}L \sim_{\Qbb} -\gamma E. \]
Here the first equivalence is essentially due to the fact that the canonical form on $X$ has pole of order $1$ along $E$. 
But we also have $\text{Pic}(Y) = 0$, so if $K_{Y,D}$ is $\Qbb$-Cartier then actually $ 0 \sim_{\Qbb} K_{Y,D}$, hence
\[ K_{\wt{Y}, \wt{D}} + (1+\gamma)E \sim_{\Qbb} 0 \sim_{\Qbb} p^{*} K_{Y,D}. \]
It follows that $\text{discrep}(Y,D) = \min \set{-1-\gamma, \text{discrep} (\wt{Y}, \wt{D})}$. Since $(\wt{Y}, \wt{D}) \to (X,B)$ is a surjective smooth morphism, we also have the equality 
\[ \text{discrep}(\wt{Y},\wt{D}) = \text{discrep}(X,B), \]
cf. \cite[(2.14)]{Kol13}. 
The assertions about singularities of $(Y,D)$ then follow. 

Finally, if $K_{Y,D}$ is $\Qbb$-Cartier then $K_{Y,D} \sim_{\Qbb} 0$, so $K_Y$ is $\Qbb$-Cartier iff $D$ is $\Qbb$-Cartier (hence $D \sim_{\Qbb} 0$). In this case, we have $a(E,Y) = a(E,Y,D)$ for every exceptional divisor $E$ such that $\text{center}_Y(E) \subsetneq \text{supp}(D)$. We then conclude about the singularities of $Y$. 
\end{proof}

\begin{rmk}
When taking quasi-regular quotient, the underlying algebraic variety $X$ and the branch divisor $B$ vary with the Reeb vectors; i.e. the orbifold structure on $X$ and the $\Cbb^*$-action on $L$ vary. However, the total space of the line bundle $L$ stays the same.
\end{rmk}

\begin{ex}
Recall that a weighted projective space $\Pbb^n (w)$ is defined as the quotient $(\Cbb^{n+1}-\set{0})/ \Cbb^{*}$, where $w = (w_0,\dots,w_n)$ is a sequence of positive integers and the $\Cbb^{*}$-action on $\Cbb^{n+1}$ is 
 \[ \lambda(z_0,\dots,z_n) = (\lambda^{w_0}z_0,\dots,\lambda^{w_n} z_n).\] 
Any polarized normal affine cone $(Y,T, \xi)$ can be seen as an affine subvariety of $\Cbb^N$, where $T$ is faithfully represented a subtorus of $GL(N,\Cbb)$ and the $\xi \in \mathfrak{t}$ generates the $\Cbb^{*}$-action as above. Moreover any quasi-regular quotient is a weighted projective subvariety of $\Pbb^n(w)$. 

As a complex variety, $\Pbb^n(w)$ is isomorphic to $\Pbb^n / (\Zbb_{w_1} \times \dots \times \Zbb_{w_n})$, but for all $w$ the total spaces $L \backslash \Pbb(w)$ are isomorphic to $\Cbb^{n+1}-\set{0}$ (albeit endowed with different $\Cbb^{*}$-actions), hence independent of the polarization $\xi$. 
\end{ex}

\begin{ex}
Let $(X,L)$ be a projective variety with trivial canonical line bundle. Then $C(X,L)$ is always $\Qbb$-Gorenstein because $K_X \sim 0$. If moreover $X$ has log canonical singularities, then $Y$ is always log canonical but $Y$ is never klt since $\gamma = 0$.
\end{ex}

\subsection{Automorphisms group of affine cones}
We provide some results relating some automorphisms group of $Y$ to a quasi-regular quotient. They will be useful in the proof of Theorem \ref{theorem:main}. 

Let $(Y,\xi_0)$ be the normal affine cone polarized by a quasi-regular Reeb vector $\xi_0$, and $(X,B;L)$ the corresponding normal polarized pair. 
We denote by $Aut(Y,\xi_0)$ the automorphisms of $Y$ commuting with $\xi_0$ and $Aut(X,B;L)$ the automorphisms of $X$ preserving $B$ and $L$. 

\begin{lem} \label{lemma:automorphism}
There is an exact sequence of groups 
\begin{equation}
 1 \to \Cbb^{*} \to Aut(Y,\xi_0) \to Aut(X,B;L) \to 1.     
\end{equation}
In particular, the Reeb cone of $Y$ is one-dimensional iff $Aut(X,B;L)$ is finite. 
\end{lem}

\begin{proof}
 Let $\Cbb^{*}$ be the action generated by $\xi_0$ on the line bundle $L$, and $Aut^{\Cbb^{*}}(L)$ be the automorphisms of the variety $L$ that commutes with this $\Cbb^*$-action. Then by \cite[Lemma 3.4.1]{Bri18}, there is an exact sequence of groups 
 \begin{equation}
 1 \to \Cbb^{*} \to \text{Aut}^{\Cbb^{*}}(L) \to \text{Aut}(X,B;L) \to 1.     
\end{equation}
We conclude by the fact that there is an obvious isomorphism between $Aut^{\Cbb^{*}}(L)$ and $Aut(Y,\xi_0)$. 
\end{proof}

\begin{cor} \label{lemma:iitaka_finite_group}
Let $(X,B)$ be a variety of log general type and $Y$ is the cone with respect to the ample polarization $K_{X,B}$. Then the Reeb cone of $Y$ is one-dimensional.    
\end{cor}

\begin{proof}
By Lemma \ref{lemma:automorphism}, it is enough to show that 
\[ \text{Aut}(X,B) = \text{Aut}(X,B;K_{X,B})\] 
is finite.  This is a result due to Iitaka \cite[Theorem 11.12]{Iit82} which states that a variety $V$ with maximal logarithmic Kodaira dimension \footnote{When $X$ is the compactification of a smooth variety $V$ with smooth boundary $B$,  the logarithmic Kodaira dimension of $V$ is defined to be the Kodaira dimension of $X$ with respect to the sheaf $K_{X,B}$ \cite[Definition, p. 326]{Iit82}. Because this is a birational invariant, the definition is independent of the compactification and applies as well for singular $V$. } has finite strictly birational group $\text{SBir}(V)$ (see \cite[\S 2.12]{Iit82} for a definition); the latter contains the automorphisms group $\text{Aut}(V)$. Applying this result to $V = X \backslash B$ (whose logarithmic Kodaira dimension is maximal since $K_{X,B}$ is ample) and  using the obvious inclusion $\text{Aut}(X,B) \subset \text{Aut} (X\backslash B) \subset \text{SBir}(X \backslash B)$, we conclude that $\text{Aut}(X,B)$ is finite.
\end{proof}


\subsection{K-stability of affine cones}
In general, the K-stability notion for cones is made to include the K-stability for irregular and quasi-regular polarizations. Instead of using Riemann--Roch-type theorems \cite{RT11}, the Donaldson--Futaki invariant in this setting is defined via the \emph{index} and \emph{weight characters} \cite{CS18, MSY08}; cf. Definitions \ref{definition:index_character}, \ref{definition:weight_character}. These objects make sense for irregularly polarized affine cone $(Y,\xi)$ and give back the corresponding characters for the polarized quotient pairs $(X,B;L)$ when $\xi$ is quasi-regular. A convenience of passing to the index character is that the cyclic terms in the relevant orbifold Riemann--Roch theorems due to Toën (cf. \cite{RT11}) do not appear. 

\subsubsection{Index and weight characters}

Given a normal affine cone $(Y,T)$, there is an embedding $Y \subset \Cbb^{N}$ defined by the ideal $I \subset \Cbb[Z_1, \dots, Z_{N}]$ and a faithful representation $T \to \textnormal{GL}(N,\Cbb)$ such that the image of $T$ is a diagonal subgroup acting holomorphically and effectively on $Y$ (cf. \cite{Sum74, vC11}). Let $R = \Cbb[Z_1,\dots,Z_N]/I$ be the ring of regular functions of $Y$. Then $R$ inherits an action of $T^n$ and has a weight decomposition as a $T^n$-module:
\begin{equation} \label{eq:weight_decomposition}
R = \bigoplus_{\alpha \in \tfrak^{*}} R_{\alpha},
\end{equation}
where each $R_{\alpha}$ is isomorphic as a $\Cbb$-vector space to the ring
$$\set{f \in R \mid \Lcal_{\xi} f= \sqrt{-1} \alpha(\xi)f, \; \forall \xi \in \tfrak}.$$ With this data, the index character is defined as follows.

\begin{defn}[{\cite[equation (1.16)]{MSY08} \& \cite[Definition 4.1]{CS18}}] \label{definition:index_character}
The (\emph{equivariant}) \emph{index character---or weighted Hilbert series,} is defined for $\xi \in \mathcal{C}_0$ and $t \in \Cbb$ with $\Re(t) > 0$ (for now as a formal series) by 
\[ F(\xi, t) = \sum_{\alpha \in \tfrak^{*}} e^{-t \alpha(\xi)} \dim R_{\alpha}. \]
In other words, this is the trace of the matrix representing the action of $\xi$ on $R[Y]$.  
\end{defn}

\begin{prop}[{\cite[Lemma 4.2 \& Theorem 4.10]{CS18}, \cite{RT11}}] \label{prop:indexcharexpansion}
The formal series $F(\xi,t)$ converges if $\xi$ is a Reeb field and $\Re(t) > 0$. Moreover, $F(\xi,t)$ as a function of $t$ has a meromorphic extension to $\Cbb$ with poles along $\Re(t) = 0$ and with a Laurent expansion near $t = 0$ given by 
\begin{equation} \label{eq:indexseriesexpansion}
F(\xi,t) = \frac{a_0(\xi)n!}{t^{n+1}} + \frac{a_1(\xi)(n-1)!}{t^{n}} + O(t^{1-n}),
\end{equation}
where $a_0(\xi), a_1(\xi)$ are smooth functions on $\Ccal_R$. 
In particular, if $(X,B)$ is the orbifold quotient pair by a quasiregular Reeb vector field $\xi$, then 
\begin{equation}
a_0(\xi) = (-K_{X,B})^{n}, \quad a_1(\xi) = \frac{(-K_{X,B}).L^{n-1}}{L^n}.     
\end{equation}
For a general $\xi$, we call $a_1(\xi)$ the slope of the polarized pair $(Y,\xi)$. 
\end{prop}

\noindent The expansion \eqref{eq:indexseriesexpansion} recovers \cite[equation (1.17)]{MSY08}. 

\begin{defn} \label{definition:weight_character}
The \emph{weight characters} is defined for $\xi \in \Ccal_R$, $t \in \Cbb$ and $\eta \in \tfrak$ by 
\[ C_{\eta}(\xi,t) = \sum_{\alpha \in \tfrak^{*}} e^{-t \alpha(\xi)} \alpha(\eta). \]
This is the trace of the infinitesimal action of $\eta$ on $R_{\alpha}$. 
\end{defn}

\begin{prop}[{\cite[Theorem 4.14]{CS18}}]
The weight character $C_{\eta}(\xi,t)$ is convergent for $\Re(t) > 0$ by the same arguments as \cite[Theorem 4.10]{CS18}. Moreover, $C_{\eta}(\xi,t)$ has a meromorphic extension to $\Cbb$ with poles along $\Re(t) = 0$ and with a Laurent expansion near $t = 0$ given by 
\begin{equation} \label{eq:weightcharexpansion}
C_{\eta}(\xi,t) = \frac{b_0(\xi)}{t^{n+2}} + \frac{b_1(\xi)}{t^{n+1}} + O(t^{-n}), 
\end{equation}
with moreover 
\begin{equation}
b_0(\xi) = \frac{-1}{n+1} D_{\eta} a_0(\xi), \quad b_1(\xi) = \frac{-1}{n} D_{\eta} a_1(\xi).    
\end{equation}
\end{prop}

\subsubsection{Donaldson--Futaki invariant and K-semistability}
We recall here the Donaldson--Futaki invariant of an affine cone. The key result in this section is Theorem \ref{theorem:conekstab_vs_weightedkstab};i.e.,  the correspondence between the test configurations of a quasi-regularly polarized cone $(Y,T,\xi)$ and those of its quotient that are equivariant under the quotient torus $T/ \sprod{\xi}$. 

\begin{defn}[Affine cones test configurations]
An affine $T$-equivariant test configuration of $(Y,T,\xi)$ consists of a
 $(T \times \Cbb^{*})$-equivariantly flat affine morphism $\pi \colon \mathcal{Y} \to \Cbb$; where
 \begin{itemize}
\item  $T$ acts fiberwise on $\mathcal{Y}$ and trivially on $\Cbb$, while $\Cbb^{*}$ acts on $\Cbb$.
\item $\pi^{-1}(0)$ is $T \times \Cbb^{*}$-equivariantly isomorphic to $Y \times (\Cbb \backslash \set{0})$.  
\end{itemize}
The test configuration is said to be trivial if $\mathcal{Y}$ is $T$-equivariantly isomorphic to $ Y \times \Cbb$. 
\end{defn}

\begin{rmk}
Under the equivariant embedding $Y \subset \Cbb^N$, a one-parameter subgroup $\lambda : \Cbb^{*} \to GL(N,\Cbb)$ defines a test configuration of $Y$ (with $T$ acting as the left linear action and $\lambda$ on the right). Conversely, every test configuration arises in this way. 
\end{rmk}


\begin{defn}[Donaldson--Futaki invariant]
Let $a_0(\xi), a_1(\xi)$ be the coefficients in the meromorphic extension \eqref{eq:indexseriesexpansion} of the index character; let $b_0, b_1$ the coefficients in the extension \eqref{eq:weightcharexpansion} of the weight character. Let $\Ycal$ be a test configuration of $Y$ with the action of $\Cbb^{*}$ given by $\eta \in\tfrak$. 
The \emph{Donaldson--Futaki invariant} of $(\Ycal, \eta)$ is defined as 
\begin{equation}
DF(\Ycal,\xi, \eta) \coloneqq \frac{a_1(\xi)}{a_0(\xi)} b_0(\xi) - b_1(\xi).     
\end{equation}
\end{defn}

\begin{rmk}
 For a given test configuration $\Ycal$ with infinitesimal generator $\eta$, the function $\xi \in C_R \to DF(\Ycal,\xi, \eta)$ is a smooth function. This follows from smoothness of $a_i$ and $b_i$.   
\end{rmk}

\begin{defn}[K-stability]
The polarized cone $(Y,T,\xi)$ is said to be
\begin{itemize}
\item \emph{K-semistable} if $DF(\Ycal,\xi) \geq 0$ for all test configurations.
\item \emph{K-polystable} if it is K-semistable and $DF(\Ycal, \xi) = 0$ only when $\Ycal$ is product.
\item \emph{K-stable} if $DF(\Ycal, \xi) > 0$ for all non-product test configurations.
\end{itemize}
When $(Y,T,\xi)$ is K-semistable, the Reeb vector is called a \emph{K-semistable Reeb vector}. 
\end{defn}

\begin{lem} \label{lemma:uniqueness_reeb_fano}
For a Fano cone singularity, the K-semistable Reeb vector is unique.     
\end{lem}

\begin{proof}
By the $\Qbb$-Gorenstein klt assumption, $a_1/a_0$ extends to $\tfrak$ as a linear function, hence we can suppose that $\xi$ lies on a hyperplane $\Sigma$ \cite[Proposition 6.4]{CS19}. 
Under this normalization and evaluating the trivial test configuration at a vector tangent to $\Sigma$, we see that the K-semistable Reeb vector $\xi$ is the critical point of $a_0(\xi)$ on $\Sigma$, which is unique since $a_0$ is strictly convex by the formula $ a_0(\xi) = \lim_{t \to +\infty} t^{n+1} F(\xi,t)$ (cf. \eqref{eq:indexseriesexpansion}).     
\end{proof}

\begin{rmk}
We will see in the proof of Theorem \ref{theorem:main} that the K-semistable Reeb vector is actually always unique.  
\end{rmk}

The K-stability notion for a pair $(Y,\xi)$ is made to be equivalent to K-stability the orbifold quotient when $\xi$ is quasi-regular. 

\begin{thm} \label{theorem:conekstab_vs_weightedkstab}
 Let $(Y,T,\xi_0)$ be a $\Qbb$-Gorenstein normal affine cone with a quasi-regular Reeb vector and $(X,B;L)$ be the quasi-regular quotient pair of $Y$ with respect to $\xi_0$ so that $K_{X,B} = \gamma L $. Let $T_{0} \coloneqq T/ \sprod{\xi_0}^{\Cbb}$ be the torus action induced on $(X,B)$. 
 Then there is a one-to-one correspondence between 
 \begin{itemize}
     \item The test configurations of $(Y,T,\xi_0)$;
     \item The $T_{0}$-equivariant test configurations of $(X,B;L)$. 
 \end{itemize} 
 Moreover for a corresponding pair of test configurations $(\Ycal, \xi_0, \eta)$ and $(\Xcal,\Bcal, \Lcal)$, we have $DF(\Ycal,\xi_0, \eta) = DF(\Xcal, \Bcal, \Lcal)$. 
\end{thm}

\begin{proof}
Let $\Ycal$ be a test configuration of $(Y,T,\xi)$ with the new $\Cbb^{*}$-action; then the scheme-theoretic quotient $(\Ycal - \ol{\Cbb^{*}(0,1)}
) /\Cbb^{*}$ defines a test configuration of $(X,B;L)$ (here $0$ is the vertex of $Y$). Conversely, given a configuration test $(\Xcal, \Bcal, \Lcal)$, then taking the fiberwise cone with respect to the fiberwise polarization of $\Lcal$ yields $\Ycal$. More precisely, let $\Fcal$ be the filtration of $(X,B)$ associated to the test configuration $(\Xcal, \Bcal; \Lcal)$; then 
\[ \Ycal = \text{Spec}_{\Cbb[t]} \bigoplus_{m \in \Zbb} \bigoplus_{\lambda \in \Zbb} t^{-\lambda} \Fcal^{\lambda}R_m. \]
The equality between Futaki invariants follows from the definition of $DF(\Ycal, \xi_0, \eta)$, which equals the Futaki invariant of its quasi-regular quotient $(\Xcal, \Bcal; \Lcal)$. 
\end{proof}

\begin{rmk}
From the construction, we have that 
\[ DF(\Xcal, \Bcal, \Lcal) = DF(\Xcal_{\xi_0}, \Bcal_{\xi_0}, \Lcal).\] 
Namely, varying the Reeb vector in a given test configuration of $Y$---with total space $(\Ycal,\eta)$, amounts to varying the orbifold structures of $(\Xcal, \Bcal)$ in the test configurations of $(X,B;L)$ while \emph{fixing the total space} of $\Lcal$. This remark will be implicitly use in Section \ref{section:proof_main}.  
\end{rmk}

\section{Proof of the Main Theorems} \label{section:proof_main}

\subsection{Proof of Theorem \ref{theorem:main}}
In this section, the cone $Y$ is assumed to be normal $\Qbb$-Gorenstein. A normal polarized pair $(X,B; L)$ consists of  a normal projective variety $X$, a $\Qbb$-divisor $B = \sum a_i B_i$ with $0 \leq a_i \leq 1$, and an ample line bundle $L$ over $X$. By a $T_0$-action on the polarized pair, we mean an effective regular action of an algebraic torus $T_0$ on $X$ preserving $B$ and linearizing $L$. 
Recall that

\begin{defn}[{\cite{OX12}}]
A \emph{log canonical model} of the pair $(X,B)$ is a birational projective morphism $f \colon Z \to (X,B)$ such that, denoting by $E_{red}$ the sum of $f$-exceptional prime divisors with coefficient $1$,
\begin{itemize}
    \item $(Z,E_{red})$ is log canonical. 
    \item $K_Z + E_{red} + B'(\coloneqq f_{*}^{-1}B)$ is ample over $X$.
\end{itemize}
\end{defn}

The existence (and uniqueness) of such model is proved by Odaka--Xu \cite{OX12}; and has since been used in a crucial way to construct destabilizing test configurations \cite{Oda13, BHJ17}. To use the model in the cone setting, we first need the following lemma. 

\begin{lem} \label{lemma:equivariant_lc_model}
Let $(X,B; L)$ be a normal polarized pair with a torus action $T_0$ such that $K_{X,B} = \gamma L$ is $\Qbb$-Cartier. Then there is a unique $T_0$-equivariant log canonical model $f \colon Z \to (X,B)$. Moreover, $\text{Ex}(f) \subset Z$ is of pure codimension $1$ and $a_i < -1$ for every component of the exceptional divisor $E_i$.
\end{lem}

\begin{proof}
We briefly explain the proof of \cite{OX12} with ingredients from the $G$-equivariant MMP (see \cite[Example 2.18]{KM98} for a discussion, and \cite{Pro21} for a more detailed treatment). The main steps consist of running the MMP from a log resolution of $\wt{Y} \to (X,B)$. Namely, 
\begin{itemize}
    \item construct a dlt modification $(Z,\Delta_Z)$ of $(X,B)$ (cf. \cite[(2.44)]{KM98} for a definition of dlt due to Sz\'{a}bo). $(Z,\Delta_Z)$ can be built by running a sequence of $(K_{\wt{Y}} + \Delta_{\wt{Y}})$-MMP over $(X,\Delta)$ \cite[Section 4.1]{Fuj11};
    \item show that there is a good minimal model $(Z',\Delta_{Z'})$ of $(Z,\Delta_Z)$ over $X$ (here good minimal model means that $Z'$ is a minimal model such that a multiple of the relative canonical ring associated to $K_{Z'} + \Delta_{Z'}$ is finitely generated). This can again be done by a sequence of MMP \cite[Lemma 2.6]{OX12};
    \item take $m_0$ large enough so that $m_0(K_{Z'} + \Delta_{Z'})$ is Cartier, then the relative log canonical model of $Z'$ over $X$ 
    \[ Y' \coloneqq \text{Proj} R(Z'/X, K_{Z'} +\Delta_{Z'}),\] 
    turns out to be the log canonical model of $(X,B)$ \cite[Lemma 2.2]{OX12}. 
\end{itemize}

Now since $(X,B)$ admits an algebraic group action $G$, we can take a $G$-equivariant log resolution $\wt{Y}$ of $(X,B)$ (e.g., functorial resolution \cite[(9.1)]{Kol07}). Then it only remains to show that the birational modifications in the MMP runs $G$-equivariantly starting from a log canonical pair. Namely, assuming first that $G$ is finite, one can take the following steps :
\begin{itemize}
    \item establish a $G$-invariant version of the (relative) Cone Theorem; that is, for the classes of $G$-invariant numerically effective curves $NE(.)^G$ we have 
    \[ \ol{NE}(.)^G = \ol{NE}(.)^G_{K_{(.)} \geq 0} + \sum \Rbb^{+} [C_i],\]
    where $[C_i]$ are countably many extremal rays in $\ol{NE}(.)^G$. This is asserted in \cite[Theorem 3.3.1]{Pro21};
    \item show that for each curve in the negative part of $K_{(.)}$, $G$-equivariant contraction and flip exist at each step provided that the original pair is lc; see \cite[Theorem 3.4.3]{Pro21} for a proof;
    \item show that if the sequence of $G$-equivariant flips and contractions is finite for $G = \set{1}$, then it is finite for a general $G$ \cite[Proposition 4.4.2]{Pro21}. 
\end{itemize}
 At each step (e.g, taking a flip), the group $G$ might act only rationally; however it is well-known that one can regularize the action; cf. \cite[Theorems 15.2.1, 15.2.2]{Pro21} and references therein.
 
 Reduction from the general case to the finite case follows from the arguments in \cite[525]{Pro21}, which show that the connected identity component of an algebraic group  $G$ acts trivially on the Mori cone $\ol{NE}(.)$.

 Finally for such a model $f \colon Z \to X$ we have 
 \[ K_{Z} + E_{\text{red}} + B' \sim_{\Qbb} f^* (K_X +B) + \sum_i (a_i +1) E_i, \; E_{\text{red}} \coloneqq \sum_i E_i. \]
By definition $\sum_{i} (a_i + 1) E_i$ is $f$-ample, hence $a_i < -1$ by the Negativity Lemma; cf. \cite[(3.39)]{KM98} or \cite[Lemma 1.13 (i)]{BHJ17}.
\end{proof}

We then obtain the following result. Although we will only use the conclusion in the case where $(X,B)$ is a log Fano pair, we still record the full statement here for the sake of generality. 

\begin{prop} \label{proposition:nonKss_test_configuration}
Let $(X,B; L)$ be a normal polarized pair with a $T_0$-action. 
If $(X,B)$ is not log canonical, then there is a non-K-semistable $T_0$-equivariant test configuration of $(X,B)$. Namely, there is a sequence of $T_0$-equivariant test configurations $(\Xcal, \Bcal, \Lcal'_r)$ such that 
\[ DF(\Xcal, \Bcal, \Lcal'_r) \to -\infty, \; r \to +\infty. \]  
The same conclusion holds when $(X,B)$ is a log Fano pair which is not klt. 
\end{prop}

\begin{proof}
The proof follows the same strategy as in \cite{Oda13} and crucially uses existence of the equivariant log canonical model in Lemma \ref{lemma:equivariant_lc_model}. We reproduce the arguments here for the reader's convenience. The gist of the proof is a construction of an ideal sheaf $\Ical$ on $X \times \Cbb$, satisfying the following properties
\begin{enumerate}
\item \label{property_support-invariant} $\Ical$ is  $T_0 \times\Cbb^{*}$-invariant (where $\Cbb^{*}$ acts on $\Cbb$) and supported on the subscheme $X \times \set{0}$ (this is called a flag ideal sheaf in \cite{Oda13}). 
\item \label{property_destabilizing} $\Ical$ is destabilizing; i.e., the deformation to the normal cone of $\Ical$  yields a non-K-semistable test configuration, provided $(X,B)$ is not log canonical. 
\end{enumerate}
 To simplify the proof, we work with the pair $(X,0)$ instead of the full pair $(X,B)$ since no substantial difficulties arise (see \cite[Theorem 3.7]{OS15} for the intersection formula of the Donaldson--Futaki invariant of a pair). 
 
\noindent \emph{Step \eqref{property_support-invariant}: Construction of the ideal sheaf $\Ical$ and corresponding test configuration.}

The ideal sheaf $\Ical$ is obtained in an ad-hoc manner using the $T_0$-equivariant log canonical model $Z$ built in Lemma \ref{lemma:equivariant_lc_model}.
First of all, since $Z \to X$ is a birational projective morphism we can take
\[I := \pi_{*}\Ocal_Z(e(K_{Z} +E_{red})), \]
for a suitable $e$ so that $Bl_I X \simeq Z$ (cf. \cite[Chapter II]{Har77}). The exceptional divisors $E_i$ of this blow-up satisfies 
\begin{equation}
 a(E_i, X) < -1,
\end{equation}
by Proposition \ref{lemma:equivariant_lc_model}. 
The ideal sheaf $\Ical$ defined as
\begin{equation}
 \Ical := \ol{(I+(t^m))^N},
\end{equation}
 clearly satisfies property \eqref{property_support-invariant}. The total space $\Xcal := Bl_{\Ical} (X\times \Cbb)$ is then a normal variety and inherits a $T_0 \times \Cbb^{*}$-action. Geometrically, $\Xcal$ is obtained via a base change $t \to t^m$ and the normalization of the deformation to the normal cone $Bl_{I +(t)} (X \times \Cbb)$. We thus obtain a $T_0$-equivariant degeneration via the composed map 
\begin{equation*}
 \Xcal \to X \times \Cbb \to \Cbb.    
\end{equation*}
Set $\Lcal \coloneqq p_1^{*}L$ where $p_1$ is the projection $X \times \Cbb \to X$; denote by $E' \coloneqq \sum a_i(E_i', X \times \Cbb) E_i'$
the exceptional divisor of the blow-up $\Xcal \to X \times \Cbb$. Then for some integer $r$ large enough, the line bundle 
\[ \Lcal'_{r} \coloneqq r\Pi^{*}\Lcal - E'\] 
is relatively ample on $\Xcal \to \Cbb$ and (possbily after taking a multiple) $T_0 \times \Cbb$-linearized. Moreover, $\Lcal'_r$ restricts to $L$ on every non-central fiber; hence, we obtain a $T_0$-equivariant test configuration $(\Xcal, \Lcal'_r)$ for some large $r$.

To make sense of intersection numbers, we will need to consider the completion $(\ol{\Xcal}, \ol{\Lcal})$ of $(\Xcal, \Lcal)$ by adding the divisor $\infty$ from $ X \times \Pbb^1$. This is but the blowup 
\begin{equation} \label{equation:projective_testconfiguration}
\Pi \colon \ol{\Xcal} \to X \times \Pbb^1
\end{equation}
along $\Ical \times \set{0}$; the exceptional divisor of which will still be denoted by $E'$. We also denote by 
\[ \ol{\Lcal'_r} \coloneqq r \Pi^{*} \ol{\Lcal} - E'. \]   \\
\noindent \emph{Step \eqref{property_destabilizing}: Destabilizing property of the ideal sheaf $\Ical$.}

This is achieved by computing the asymptotic of the Donaldson--Futaki invariant $DF(\Xcal, \Lcal'_{r})$ as $r \to +\infty$ via the following intersection-theoretic formula due to Odaka and Wang \cite{Oda13b, Wan12}:
\begin{equation} \label{equation:futaki_intersection}
\begin{aligned}
DF(\Xcal, \Lcal') &= -n(L^{n-1}.K_X)(\ol{\Lcal'})^{n+1} \\
&+ (n+1)L^n \tuple{(\ol{\Lcal'})^n. \Pi^*p_1^{*}K_X} \\
&+ (n+1)L^n \tuple{(\ol{\Lcal'})^n.K_{\ol{\Xcal}/ X \times \Pbb^1}}.
\end{aligned}
\end{equation}
More precisely, as $r \to +\infty$, $DF(\Xcal, \Lcal'_{r})$ is a polynomial in $r$ with leading coefficient \footnote{The coefficient $S_{(X,L)}(\Ical)$ is called the $S$-coefficient in \cite{Sha81} and \cite[Proposition 3.4]{Oda13} (see also \cite{BHJ17}, where this is baptized the non-Archimedean entropy).}
\begin{equation} \label{equation:asymptotic_futaki}
c_{n,d} S_{(X,L)}(\Ical) r^{n+d } + O(r^{n+d-1}); \quad S_{(X,L)}(\Ical) \coloneqq (\ol{\Lcal}^d(\ol{\Lcal'})^{n-d}. K_{\ol{\Xcal}/ X \times \Pbb^1}),
\end{equation}
where $d = n-1$ is the dimension of $\text{Supp}(\Ocal_{X \times \Cbb} / \Ical)$ (by the pure-dimensional assertion in Lemma \ref{lemma:equivariant_lc_model}), and $c_{n,d} > 0$. To conclude, it is sufficient to show that 
\begin{equation*} 
S_{(X,L)}(\Ical) < 0, 
\end{equation*}
provided $(X,B)$ is not log canonical. This follows from a positivity property of intersection numbers (cf. \cite[Lemma 3.5]{Oda13}); namely, 
\begin{equation} \label{equation:positivity_intersection_number}
S_{(X,L)}(\Ical) = \Pi^{*}D^d. (D')^{n-d}.(-E) < 0,
\end{equation}
where
\begin{itemize}
\item $D$ is any ample Cartier divisor $H \times \Pbb^1$ with $H \in \abs{L}$,
\item  $D'$ any ample Cartier divisor in $\abs{r \ol{\Lcal'}}$, 
\item $E$ is the effective divisor $-K_{\ol{\Xcal} / (X \times \Pbb^1)}$. Here effectiveness follows from a computation in local coordinates and non-log-canonicity of $(X,B)$:
\begin{equation} 
a(E_i', X \times \Cbb) = c_i(a_i(E_i,X) + 1) < 0;
\end{equation}
cf. \cite{Oda13}.  
\end{itemize}
Finally using \eqref{equation:asymptotic_futaki}, \eqref{equation:positivity_intersection_number}, we conclude that $DF(\Xcal, \Lcal'_r) \to -\infty$ for $r \to +\infty$, i.e. $(X,B;L)$ is not K-semistable if $(X,B)$ is not log canonical. 

To conclude the log Fano case ($L = -K_{X,B}$), assume that $(X,B)$ is K-semistable log canonical but not klt. Let $\pi \colon \wt{X} \to X$ be a $T_0$-equivariant log resolution of $X$ with reduced exceptional divisor $E_{\text{red}} = \sum E_i$. Then $(X', (1-\varepsilon)E + B')$ is a klt pair for $\varepsilon$ small enough.  Take $m_0$ large enough so that $D_0 = m_0\sum (A_i - \varepsilon)E_i$ is a Cartier divisor, and set
\[ (R(X'/X, D_0)) = \bigoplus_m \pi_{*}\Ocal_{X'}(D_0). \]
Then the latter is finitely generated \cite[Corollary 1.1.2]{BCHM10}. Moreover, $\text{Proj} R(X'/X,D_0)$ defines a
 $T$-equivariant log canonical model $Z \to X$ which is actually klt \cite[Theorem 1.2]{BCHM10}, hence a projective morphism which is not isomorphism. Moreover, $Z \to X $ is the blow-up of some $T_0$-invariant ideal sheaf; and we can construct the test configuration by  \eqref{equation:projective_testconfiguration}. In this case $K_{\ol{\Xcal}/ (X \times \Pbb^1)}$ is trivial since $(X,B)$ is non-klt. Moreover, the part 
\[ -n(-K_X)^n (r\Pi^*(\ol{-\Kcal_{\Xcal}}) - E')^{n+1}  \]
in the intersection formula \eqref{equation:futaki_intersection} goes to $-\infty$ as $r \to +\infty$; hence $(X,B)$ is not K-semistable. 
\end{proof}

We now conclude the proof of Theorem \ref{theorem:main}.\\ 
\noindent \emph{K-semistability implies log canonical.} \\
 Fix a K-semistable polarization $(Y,T,\xi)$ and assume that  $Y \backslash \set{0}$ is not log canonical; we will show that $(Y,T,\xi)$ has negative Donaldson--Futaki invariant for some test configuration. 
 
 We consider a GIT quotient pair $(X,B)$ of $Y$ by a quasi-regular Reeb vector $\xi_0 \in \tfrak $. Since $Y$ is $\Qbb$-Gorenstein, there is an ample line bundle $L \to X$ such that $K_{X,B}= \gamma L$ is $\Qbb$-Cartier. 
By assumption $(X,B)$ is  non-log-canonical pair with a $T_0$ action, so there is a $T_{0}$-equivariant configuration test $(\Xcal, \Bcal, \Lcal'_r)$ with $DF(\Xcal, \Bcal, \Lcal'_r) < -1$ by Proposition \ref{proposition:nonKss_test_configuration}. 
Hence, by Theorem \ref{theorem:conekstab_vs_weightedkstab}, $(\Xcal, \Bcal, \Lcal'_r)$ corresponds to a $T$-equivariant configuration test $(\Ycal, \eta)$ of $Y$ with 
 \[ DF(\Ycal,\xi_0, \eta) = DF(\Xcal, \Bcal, \Lcal'_r) < -1. \]
 By choosing $\xi_0$ quasi-regular so that $\mathfrak{t}$ contains both $\xi, \xi_0$, it follows by continuity of $\xi \to DF(\Ycal, \xi,\eta)$ that $DF(\Ycal, \xi, \eta) \leq -1 < 0$; i.e. $(Y,\xi)$ is not K-semistable. \\

 \emph{The case $a_1(\xi) < 0$.}  We will prove that the Reeb cone of $Y$ is in fact one-dimensional.
 By continuity of $a_1$ there is a  quasi-regular Reeb vector $\xi_0$ such that $a_1(\xi_0)  <0$. Then by Theorem \ref{theorem:cone-base-singularities} the $\xi_0$-quasi-regular quotient $(X,B)$ has ample log canonical bundle $K_{X,B}$, hence the Reeb cone of $Y$ is one-dimensional (cf. Corollary \ref{lemma:iitaka_finite_group}). In particular, 
\begin{align*}
Y  \; \text{K-stable} &\iff (X,B) \; \text{K-stable}\\
&\iff  (X,B) \; \text{log canonical} \; 
\text{\cite[Theorem 4.1]{OS15}} \\
&\iff (X,B) \; \text{K-semistable} \\
&\iff Y\backslash\set{0} \; \text{log canonical},
\end{align*}
which allows us to conclude.
 \\

 \emph{The case $a_1(\xi) > 0$.} 
Assume that $(Y,T,\xi)$ is K-semistable but $Y \backslash\set{0}$ is not klt; then any quasi-regular quotient $(X,B)$ is not klt. Since $a_1(\xi) > 0$, we can choose $\xi_0$ close enough to $\xi$ such that $a_1(\xi_0) > 0$ (by continuity of $a_1$), i.e., the corresponding quotient $(X,B)$ is a log Fano pair. Then non-klt implies by Proposition \ref{proposition:nonKss_test_configuration} that there is a non-trivial $T_0$-equivariant test configuration of $(X,B)$, which lifts to a non-trivial test configuration $(\Ycal, \eta)$ such that \[ DF(\Ycal,\xi_0,\eta) \leq -1. \]
Letting  $\xi_0$ go to the K-semistable $\xi$, we have 
\[DF(\Ycal,\xi,\eta) \leq -1< 0,\] 
that is, $Y$ is not K-semistable, a contradiction. Thus, $Y \backslash \set{0}$ must be klt. 

Finally, since any quasi-regular quotient $(X,B)$ by $\xi_0$ close enough to $\xi$ is klt with $K_{X,B} = \gamma L$ and $\gamma = -a_1(\xi_0) <0$, it follows from Theorem \ref{theorem:cone-base-singularities} that $Y$ is itself klt. \\

 \emph{The case $a_1(\xi) = 0$.} 
Assume $(Y,T,\xi)$ is K-stable with $Y \backslash \set{0}$ non-klt. We first show that in this case $\xi$ must be quasi-regular. If in a neighborhood of $\xi$ we have $a_1(\xi_0) > 0$ (for some $\xi_0 \neq \xi$), then for every such quasi-regular $\xi_0$, the quotient $(X,B)$ is log Fano without being klt, hence $(X,B)$ is not K-semistable. The same argument as above shows that $(Y,T,\xi)$ is not K-stable, a contradiction. It follows that $a_1(\xi_0) \leq 0$ in a neighborhood of $\xi$, but if $a_1(\xi_0) < 0$ for some quasi-regular $\xi_0$ the Reeb cone of $Y$ is one-dimensional cf. Lemma \ref{lemma:iitaka_finite_group}, so we also have $a_1(\xi)  < 0$, a contradiction.

Thus $a_1(\xi_0) = 0$ in a neighborhood of $\xi$; so $Y$ is actually a cone over a compact log Calabi--Yau variety $(X,B)$ which is not klt since $Y \backslash \set{0}$ is not klt. Using \cite[Theorem 4.1]{OS15}, we see that for every quasi-regular $\xi_0$ in the neighborhood, the corresponding quotient $(X,B)$ is in fact not K-stable. We can then find a non-trivial $T_0$-equivariant test configuration of $(X,B)$ corresponding to a non-trivial test configuration $(\Ycal,\eta)$ of $Y$ such 
that 
\[ DF(\Ycal, \xi_0, \eta) \leq 0.\]
Choosing a quasi-regular sequence $\xi_0$ going to $\xi$ while fixing  $(\Ycal,\eta)$, we have that 
\[ DF(\Ycal, \xi, \eta) \leq 0, \]
a contradiction. We conclude that $Y \backslash \set{0}$ is klt if $(Y,T,\xi)$ is K-stable. 

Conversely, assume that $Y \backslash \set{0}$ is klt.  If there is a quasi-regular $\xi_0$ such that $a_1(\xi_0) > 0$, then the quotient $(X,B)$ is log Fano with klt singularities; i.e. $Y$ is a Fano cone singularity. But a K-semistable Reeb vector is unique and satisfes $a_1(\xi) > 0$ in this case, a contradiction. It follows that 
\[ a_1(\xi_0) \leq 0 \]
for every quasi-regular $\xi_0$, hence any quasi-regular quotient is either a klt log Calabi--Yau pair or a klt log pair of general type, both are K-stable by \cite[Theorem 4.1]{OS15}. In either case $\text{Aut}(X,B)$ is finite  (for a log pair of general type, see Lemma \ref{lemma:iitaka_finite_group}; for a log Calabi--Yau pair, 
see \cite{BG08, Oda12} for an orbifold pair $(X,B)$; and \cite[Theorem 5.3]{DR24} for the general case). Thus the Reeb cone of $Y$ is one-dimensional and the K-semistable Reeb vector is unique and coincides with the Reeb vector of the quotient pair (in particular $a_1(\xi_0) = a_1(\xi) = 0$). It follows that $Y$ is K-stable.


\subsection{Proof of Theorem \ref{theorem:cscK_gorenstein}}
 
Since $Y$ is assumed to be K-stable $\Qbb$-Gorenstein, $Y \backslash \set{0}$ is klt, we can define cscK metrics outside the apex in the sense of \cite{EGZ}. 
Any such metric then induces a cscS metric with curvature $n(n+1)$ on the (singular) Sasaki link with $a_1(\xi) > 0$. By continuity of $a_1$, we can choose $\xi_0$ rational close to $\xi$ such that $a_1(\xi_0) > 0$, with $(X,B)$ as the corresponding quotient of $Y$ by $\xi_0$, polarized by $L$. 

Again since $Y$ is $\Qbb$-Gorenstein we have $K_{X,B} = \gamma L$, where $\gamma = -a_1(\xi_0)$ since $a_1(\xi_0) = (-K_{X,B}) L^{n-1} / L^n$, cf. Proposition \ref{prop:indexcharexpansion}. This together with the K-stability assumption implies that $(X,B)$ is klt (by Theorem \ref{theorem:main}) with moreover $\gamma <  0$, hence from Theorem \ref{theorem:cone-base-singularities} $Y$ is klt. 

Thus, $Y$ is a K-stable Fano cone singularity, but K-stability of such cone implies that $(Y, \xi)$ has Ricci-flat Kähler cone metric with the same Reeb vector $\xi$, which is necessarily unique \cite{MSY08, CS18}. It follows from uniqueness of cscK cone metric with respect to $\xi$ that there is an automorphism $\phi$ of $(Y,\xi)$ sending the Ricci-flat metric to the cscK cone metric while fixing the Reeb vector and complex structure. This terminates our proof. 

\printbibliography
\end{document}